\documentclass{amsart}
\usepackage{amssymb}
\usepackage{latexsym}
\usepackage{amsmath}
\usepackage{enumerate}

\newtheorem{theorem}{Theorem}[section]
\newtheorem{lemma}[theorem]{Lemma}
\newtheorem{corollary}[theorem]{Corollary}
\newtheorem{proposition}[theorem]{Proposition}

\theoremstyle{definition}
\newtheorem{definition}[theorem]{Definition}

\numberwithin{equation}{section}

\title{The $\mathsf{HOD}$ Hypothesis and a supercompact cardinal}
\author{Yong Cheng}
\address{School of Philosophy,
Wuhan University, Wuhan, Hubei Province, P.R.China, 430072}
\email{world-cyr@hotmail.com}
\thanks{I am indebted to W. Hugh Woodin for our communications about the idea of his work, his patience to
answer my questions about the HOD Hypothesis, and in particular the discussions of his Local Universality Theorem. I should
like to thank the referees for their helpful suggestions and corrections for improvement, David Asper\'{o} for suggestions on the
improvement of Section 2, the Chinese Ministry of Education for their support via the Humanities and Social Sciences Planning
Fund (project no. 17YJA72040001) and Wuhan University in China for the support via the Luo Jia young scholar research
project and the Young Scholar Academic Team Project.}



\begin{document}

\begin{abstract}
In this paper, we prove that: if $\kappa$ is supercompact and the $\mathsf{HOD}$ Hypothesis holds, then there is a proper class of regular cardinals in $V_{\kappa}$ which are measurable in $\mathsf{HOD}$.  From \cite{Woodin 5}, Woodin also proved this result. As a corollary, we prove Woodin's Local Universality Theorem.
This work shows that under the assumption of the $\mathsf{HOD}$ Hypothesis and supercompact cardinals, large cardinals in $\mathsf{V}$ are reflected to be large cardinals in $\mathsf{HOD}$ in a local way, and reveals the huge difference between $\mathsf{HOD}$-supercompact cardinals and supercompact cardinals under the $\mathsf{HOD}$ Hypothesis.
\end{abstract}

\maketitle 
\section{Introduction}

The $\mathsf{HOD}$ Hypothesis is an important hypothesis about $\mathsf{HOD}$ proposed by W.Hugh Woodin in \cite{Woodin new}, which says that there is a proper class of regular cardinals that are not $\omega$-strongly measurable in $\mathsf{HOD}$ (see Definition \ref{strongly measurable} and Definition \ref{HOD conjecuture}). In \cite{Woodin 2}, Woodin uses the term ``The $\mathsf{HOD}$ Conjecture" to denote the same statement as the $\mathsf{HOD}$ Hypothesis. For this paper, our main references are  \cite{Woodin 2} and \cite{Woodin new}; all basic facts about the $\mathsf{HOD}$ Hypothesis we know are in \cite{Woodin 2} and \cite{Woodin new}. Our notations are standard, see \cite{Jech} and \cite{Set theory}.

Examining under which hypothesis $\mathsf{HOD}$ and $\mathsf{V}$ are close to each other and how $\mathsf{HOD}$ and $\mathsf{V}$ can be pushed apart via forcing is a very interesting area of research. From \cite{joint work with Joel}, via forcing, behaviors of large cardinals from $\mathsf{V}$ can become disordered in $\mathsf{HOD}$. A natural question is whether the $\mathsf{HOD}$ Hypothesis has some effect on the behavior of large cardinals from $\mathsf{V}$ in $\mathsf{HOD}$.  We want to know whether under the $\mathsf{HOD}$ Hypothesis, behaviors of large cardinals from $\mathsf{V}$ become more regular in  $\mathsf{HOD}$; Especially, whether and how, under the $\mathsf{HOD}$ Hypothesis, large cardinals in $\mathsf{V}$ can be transferred into $\mathsf{HOD}$. In this paper, we answer this question for one supercompact cardinal and prove the following main result: if $\kappa$ is supercompact and the $\mathsf{HOD}$ Hypothesis holds, then there is a proper class of regular cardinals below $\kappa$ which are measurable in $\mathsf{HOD}$. From \cite{Woodin 5}, Woodin also proved this theorem. Woodin proved the Global Universality Theorem in \cite[Theorem 201]{Woodin 2} and  announced his Local Universality Theorem in \cite{Woodin 5}. As a corollary of the above main result, we have Woodin's Local Universality Theorem (see Corollary \ref{Local Universality Theorem}).

This paper is organized in the following way: In Section 2, we discuss the three main motivations for the $\mathsf{HOD}$ Hypothesis; In Section 3, we give a systematical and self-contained introduction to the $\mathsf{HOD}$ Hypothesis and its basic facts which would be used in later passages; In Section 4, we prove our main result Theorem \ref{main result of the paper}; In Section 5, we conclude with some natural and interesting questions.

\section{The Motivation of the $\mathsf{HOD}$ Hypothesis}

The inner model program for one supercompact cardinal, the limits of the large cardinal hierarchy and the $\mathsf{HOD}$ Dichotomy Theorem are the three main motivations for the $\mathsf{HOD}$ Hypothesis.

\begin{enumerate}[(1)]
\item Inner model theory has a long and complex history, starting with Jensen's work on $\mathsf L$ from the 1960's. There is a large variety of inner models (by `inner models' we mean transitive models of $\mathsf{ZFC}$ containing all the ordinals), and one natural classification criterion for them is their structural simplicity and their invariance with respect to extensions of the universe via forcing. In one extreme we have $\mathsf L$. $\mathsf L$ has a well--understood fine structure and all models of set theory with the same ordinals have exactly the same version of $\mathsf L$. It follows that we can decide most natural questions in mathematics by working in $\mathsf L$ (more accurately put, by working in the theory $\mathsf{ZFC}$ + $\mathsf{V=L}$). In the other extreme, we have the universe, $\mathsf V$, which in a typical theory $T$ of the form $\mathsf{ZFC}$ + large cardinals is quite underdetermined.\footnote{There will be for example forcing extensions satisfying this same theory $T$ but disagreeing with the ground model about the truth value of, for example, the Continuum Hypothesis.}  $\mathsf{HOD}$ is also, to a large extent, such an underdetermined inner model. Given the above, it would seem that $\mathsf L$ would be a natural choice for our universe. $\mathsf L$ has a serious drawback, though, which is that it can contain only very weak large cardinals.

The main goal of inner model theory is to build, under suitable assumptions,\footnote{For example, but not only, under the assumption that the relevant large cardinal axiom holds in $\mathsf V$.} inner models containing suitable large cardinals but with as many of the nice structural properties of $\mathsf L$ as possible (in particular, it would be desirable to be able to run a `fine--structural' analysis of these models).  Also, these inner models would be typically supposed to be as small as possible (in the sense of containing, besides all ordinals, just the bare minimum amount of information that would enable them to accommodate the large cardinal hypothesis at hand). Inner models of large cardinals in this sense are always so--called \emph{extender models}, i.e., models constructed in the same way as $\mathsf L$ but incorporating in the construction certain (carefully chosen) approximations to the relevant elementary embeddings that we would like the final model to capture.  The strongest large cardinal hypothesis within reach of the present inner model theory is some accumulation of Woodin cardinals. This is much stronger than, say, the existence of a measurable, but far weaker than, for example, the existence of a supercompact cardinal.\footnote{Supercompact cardinals figure prominently in many consistency proofs in higher set theory; famously, in the consistency proof of the maximal forcing axiom Martin's Maximum due to Foreman--Magidor--Shelah in 1984, for example.}

A surprising fact due to Woodin is that if the inner model program can be extended to prove that if there is a supercompact cardinal then there is a so--called $\mathsf{L}$--like weak extender model for a supercompact cardinal, then that $\mathsf{L}$--like model accommodates all large cardinal axioms that have ever been considered\footnote{Woodin uses the term \emph{Ultimate--$\mathsf{L}$} to refer to the hypothetical inner model that includes supercompact cardinals and therefore all large cardinals.This Ultimate--$\mathsf{L}$ would be robust enough with respect to forcing that one would be able to answer essentially all natural questions by working in $\mathsf{V}=$ Ultimate--$\mathsf{L}$. This would make a very strong case for adopting the axiom $\mathsf{V}=$ Ultimate--$\mathsf{L}$. The construction of this Ultimate--$\mathsf{L}$, if possible, would be a natural culmination of the inner model program the way it is currently understood.}  and is close to $\mathsf{V}$ in a certain well--defined sense. If that construction of an $\mathsf{L}$--like model is definable (so the model is contained in $\mathsf{HOD}$), then $\mathsf{HOD}$ must necessarily be close to $\mathsf{V}$ in the relevant sense, and in particular the $\mathsf{HOD}$ Hypothesis must be true.
Therefore, if the inner model program can be extended to the level of one supercompact cardinal, then the $\mathsf{HOD}$ Hypothesis must be true (if there is a supercompact cardinal).  This motivates the $\mathsf{HOD}$ Hypothesis, in the sense that  the $\mathsf{HOD}$ Hypothesis is a good test question for the success of the inner model program for one supercompact cardinal and, by the comments above, for the happy conclusion of the inner model program.

\item If the $\mathsf{HOD}$ Hypothesis is provable, then one can in a natural hierarchy of large cardinal axioms give a threshold for inconsistency, against just $\mathsf{ZF}$ as background theory, which closely parallels Kunen's inconsistency in the $\mathsf{ZFC}$ context.

\begin{theorem}\label{Kunen theorem}
(Woodin, \cite{Woodin 2})\quad ($\mathsf{ZF}$) \quad Assume ``$\mathsf{ZFC} \,+$ there is a supercompact cardinal" implies the $\mathsf{HOD}$ Hypothesis. Suppose $\delta$ is an extendible cardinal and $\lambda > \delta$. Then there is no non-trivial elementary embedding $j: \mathsf{V}_{\lambda+2} \rightarrow \mathsf{V}_{\lambda+2}$.
\end{theorem}

It is a matter of fact that most large cardinal hypotheses can be naturally stated in terms of the existence of elementary embeddings of the form $j:(\mathsf{V}, \in)\longrightarrow (\mathsf M, \in)$ different from the identity, where $\mathsf M$ is some  transitive model. The closer the structure $\mathsf M$ is to $\mathsf V$, the stronger is the large cardinal situation posited. Usually, the relevant large cardinal is the critical point of $j$ (i.e., the least ordinal $\kappa$ such that $\kappa < j(\kappa)$). The above has been traditionally a general template for generating large cardinal axioms and explains, in many cases, why most large cardinal axioms considered to date tend to build a linearly ordered hierarchy with respect to consistency strength.\footnote{Typically, if $j:(\mathsf V, \in)\longrightarrow (\mathsf M, \in)$ is an elementary embedding with critical point $\kappa$ and $\mathsf M$ is ``sufficiently close to $\mathsf V$'',  then $\mathsf M$ thinks that $\kappa$ is the critical point of an elementary embedding $i:(\mathsf M, \in)\longrightarrow (\mathsf N, \in)$  in which the target model $\mathsf N$ is in principle ``less close to $\mathsf M$'' than $\mathsf M$ was to $\mathsf V$. By elementarily of $j$ and since $\kappa < j(\kappa)$ it follows that, in $\mathsf V$, there is a $\lambda < \kappa$ which is the critical point of an elementary embedding $i:(\mathsf V, \in)\longrightarrow (\mathsf N, \in)$, where $\mathsf N$ has the second (weaker) degree of closure relative to $\mathsf V$. } For example, $\kappa$ is a supercompact cardinal if and only if for every ordinal $\lambda$ it holds that $\kappa$ is the critical point of some elementary embedding $j:(\mathsf V, \in)\longrightarrow (\mathsf M, \in)$, where $\mathsf M$ is a transitive class closed under $\lambda$--sequences (i.e., for every sequence $(a_i\,:\,i<\lambda)$, if each $a_i$ is in $\mathsf M$, then $(a_i\,:\,i<\lambda)\in \mathsf M$). A natural upper limit for large cardinal axioms given by the above template is therefore the situation where $\mathsf M$ is actually all of $\mathsf V$; in other words, the statement that there is an elementary embedding $j: (\mathsf V, \in)\longrightarrow (\mathsf V, \in)$ which is not the identity.\footnote{The way it is expressed here, this is a second order statement, but there are various ways to make sense of this in a first order way.} The existence of such an elementary embedding was proposed by W.\ Reinhardt in his doctoral dissertation from 1967. A few years later (in 1971), Kenneth Kunen proved in a landmark result that such elementary embeddings cannot possibly exist. One hypothesis used crucially in Kunen's proof -- and in all other alternative proofs found afterwards -- is that $\mathsf V$ satisfies the Axiom of Choice. In fact, after more than $40$ years it is not yet known whether the nonexistence of a non-trivial elementary embedding $j:(\mathsf V, \in)\longrightarrow (\mathsf V, \in)$ can be proved just assuming $\mathsf V\models\mathsf{ZF}$.\footnote{As a matter of fact, no other inconsistency in the realm of large cardinal hypotheses has been discovered. It could well be that the existence of a non--trivial elementary embedding $j:(\mathsf V, \in)\longrightarrow (\mathsf V, \in)$ in the absence of choice is consistent. There could even be a rich hierarchy of consistent large cardinal hypotheses extending the hypothesis that there is such an elementary embedding, and therefore incompatible with choice; this would indicate that the Axiom of Choice eventually fails as we climb up the large cardinal hierarchy in very much the same way that $\mathsf V = \mathsf L$ fails as we climb up the large cardinal hierarchy (specifically, when we pass the $0^\sharp$ barrier).}
Theorem \ref{Kunen theorem} suggests that proving the $\mathsf{HOD}$ Hypothesis would have a huge foundational significance, in that it would provide a route to showing that there are no nontrivial elementary embeddings from $\mathsf{V}$ to $\mathsf{V}$ even if $\mathsf{AC}$ fails.

\item  The $\mathsf{HOD}$ Dichotomy says that either $\mathsf{HOD}$ is close to $\mathsf{V}$ or else $\mathsf{HOD}$ is far from $\mathsf{V}$.

\begin{theorem}\label{HOD Dichotomy}
(Woodin, $\mathsf{HOD}$ Dichotomy Theorem, Theorem 2 in \cite{Woodin 1}) \quad  Assume that $\delta$ is an extendible cardinal. Then exactly one of the following holds.
\begin{enumerate}[(a)]
  \item For every singular cardinal $\gamma>\delta$, $\gamma$ is singular in $\mathsf{HOD}$ and $\gamma^{+}=(\gamma^{+})^{\mathsf{HOD}}$.
  \item Every regular cardinal greater than $\delta$ is measurable in $\mathsf{HOD}$.
\end{enumerate}
\end{theorem}

Note that the two possible scenarios, (a) and (b), given by the $\mathsf{HOD}$ Dichotomy Theorem look indeed very different (i.e., (b) looks like a quite small subset of the logical negation of (a)). In fact, (a) says that $\mathsf{HOD}$ is close to  $\mathsf{V}$ in the way that $\mathsf{L}$ is close to $\mathsf{V}$ when $0^\sharp$ does not exist, and (b) says that $\mathsf{HOD}$ is small compared to $\mathsf{V}$ also in very much the same way that $\mathsf{L}$ is small compared to $\mathsf{V}$ when $0^\sharp$ exists.\footnote{More precisely, Jensen's Dichotomy Theorem for $\mathsf{L}$ says that exactly one of the following holds: (1) $\mathsf{L}$ is correct about singular cardinals and computes their successors correctly or (2) Every uncountable cardinal is inaccessible in $\mathsf{L}$. Theorem \ref{HOD Dichotomy} can therefore be seen as a generalization for $\mathsf{HOD}$ of Jensen's theorem for  $\mathsf{L}$.}
The $\mathsf{HOD}$ Dichotomy Theorem \ref{HOD Dichotomy} motivates the $\mathsf{HOD}$ Hypothesis: The $\mathsf{HOD}$ Hypothesis rules out possibility (b) and therefore says that only (a) can be the case and therefore $\mathsf{HOD}$ is always close to $\mathsf{V}$.
\end{enumerate}

\section{The $\mathsf{HOD}$ Hypothesis}

In this section, we give a self-contained exposition of Woodin's results about The $\mathsf{HOD}$ Hypothesis. Intuitively, the $\mathsf{HOD}$ Hypothesis just says that $\mathsf{HOD}$ is close to $\mathsf{V}$ in a certain sense.

The following Theorem \ref{key fact on forcing} is very important and we use it several times in this paper. Firstly, we list  some important facts about forcing with respect to $\mathsf{HOD}$ which are used to prove Theorem \ref{key fact on forcing}.
\begin{proposition}\label{key fact}
\begin{enumerate}[(1)]
\item (Forklore, \cite{Woodin 1})\quad If $\mathbb{P}$ is a weakly homogeneous and ordinal definable poset in $\mathsf{V}$ and $G$ is a $\mathsf{V}$-generic filter on $\mathbb{P}$, then $\mathsf{HOD}^{\mathsf{V}[G]}\subseteq \mathsf{HOD}^{\mathsf{V}}$.
\item (\cite[Lemma 4, Theorem 5]{Woodin 1})\quad  If $\kappa>\omega$ is an regular cardinal, $\mathbb{P}$ is a poset with $|\mathbb{P}|<\kappa, G$ is a $\mathbb{P}$-generic over $\mathsf{V}$, then in $\mathsf{V}[G],\mathsf{V}$ is $\Sigma_2$ definable from $\mathsf{V}\cap\mathcal{P}(\kappa)$.
  \item (Vop\v{e}nka, \cite[Theorem 6]{Woodin 1})\quad  For every ordinal $\kappa$, there exists $\mathbb{B}\in \mathsf{HOD}$ such that $\mathsf{HOD}\models \mathbb{B}$ is a a complete Boolean algebra, and for any $E\subseteq\kappa$, there exists a $\mathsf{HOD}$-generic filter $G$ on $\mathbb{B}$ such that $\mathsf{HOD}[E]\subseteq \mathsf{HOD}_{\{G\}}=\mathsf{HOD}_{\{E\}}=\mathsf{HOD}[G]$.
  \item (\cite[Theorem 15.43]{Jech})\quad  Let $G$ be generic on $\mathbb{B}$. If $M$ is a model of $\mathsf{ZFC}$ such that $\mathsf{V}\subseteq M\subseteq \mathsf{V}[G]$, then there exists a complete subalgebra $D\subseteq \mathbb{B}$ such that $M=\mathsf{V}[D\cap G]$.
\end{enumerate}
\end{proposition}

\begin{theorem}\label{key fact on forcing}
(\cite[Corollary 7]{Woodin 1})\quad  Let $\mathbb{P}\in \mathsf{OD}$ be a weakly homogeneous poset. Suppose $G$ is a $\mathsf{V}$-generic filter on $\mathbb{P}$. Then $\mathsf{HOD}^{\mathsf{V}}$ is a generic extension of $\mathsf{HOD}^{\mathsf{V}[G]}$.
\end{theorem}
\begin{proof}
Let $\kappa$ be an uncountable regular cardinal such that $|\mathbb{P}|<\kappa$. By Proposition \ref{key fact}(2), $\mathsf{V}$ is definable in $\mathsf{V}[G]$ from $S$ where \ $S=\mathcal{P}(\kappa)\cap \mathsf{V}$. In $\mathsf{V}$, let $\delta=|S|$ and $E$ be a binary relation on $\delta$ such that the Mostowski collapse of $(\delta, E)$ is $(trcl(\{S\}),\in)$. Then $\mathsf{HOD}^{\mathsf{V}}\subseteq\mathsf{HOD}^{\mathsf{V}[G]}_{\{E\}}$. By Proposition \ref{key fact}(3), there is a $\mathsf{HOD}^{\mathsf{V}[G]}$-generic filter $H$ on a Vop\v{e}nka algebra
such that $\mathsf{HOD}^{\mathsf{V}[G]}_{\{E\}}=\mathsf{HOD}^{\mathsf{V}[G]}[H]$. By Proposition \ref{key fact}(1), $\mathsf{HOD}^{\mathsf{V}[G]}\subseteq \mathsf{HOD}^{\mathsf{V}}$. Since $\mathsf{HOD}^{\mathsf{V}[G]}\subseteq \mathsf{HOD}^{\mathsf{V}}\subseteq \mathsf{HOD}^{\mathsf{V}[G]}_{\{E\}}=\mathsf{HOD}^{\mathsf{V}[G]}[H]$, by Proposition \ref{key fact}(4), $\mathsf{HOD}^{\mathsf{V}}$ is a generic extension of $\mathsf{HOD}^{\mathsf{V}[G]}$.
\end{proof}

\begin{definition}\label{strongly measurable}
(Woodin, \cite[Definition 189]{Woodin 2})\quad Let $\lambda$ be an uncountable regular cardinal. Then $\lambda$ is $\omega$-strongly measurable in $\mathsf{HOD}$ iff there is $\kappa<\lambda$ such that $(2^{\kappa})^{\mathsf{HOD}}<\lambda$ and there is no partition $\langle S_{\alpha}\mid\alpha<\kappa\rangle$ of $cof(\omega)\cap\lambda$  into stationary sets such that $\langle S_{\alpha}\mid\alpha<\kappa\rangle\in \mathsf{HOD}$.
\end{definition}

\begin{definition}\label{HOD conjecuture}
(Woodin, \cite[Definition 3.42]{Woodin new})\quad The $\mathsf{HOD}$ Hypothesis denotes the following statement: there is a proper class of regular cardinals that are not $\omega$-strongly measurable in $\mathsf{HOD}$.
\end{definition}

In Woodin's recent paper \cite{Woodin new}, the $\mathsf{HOD}$ Conjecture denotes the following statement as in Definition \ref{conjecture new}. In Woodin's old paper such as \cite{Woodin 1},\cite{Woodin 2} and \cite{Woodin 3}, the $\mathsf{HOD}$ Conjecture denotes the same statement as the $\mathsf{HOD}$ Hypothesis.

\begin{definition}\label{conjecture new}
(Woodin, \cite[Definition 3.48]{Woodin new})\quad The $\mathsf{HOD}$ Conjecture denotes the following statement: the theory $\mathsf{ZFC}\, +$ ``there exists a supercompact cardinal" proves the $\mathsf{HOD}$ Hypothesis.
\end{definition}

\begin{theorem}\label{strong mc imply mc}
(Woodin, \cite[Lemma 10]{Woodin 1})\quad Suppose $\kappa$ is $\omega$-strongly measurable in $\mathsf{HOD}$. Then $\kappa$ is measurable in $\mathsf{HOD}$.
\end{theorem}

Note that if $V=\mathsf{HOD}$, then no cardinals can be $\omega$-strongly measurable in $\mathsf{HOD}$. So $\kappa$ is measurable in $\mathsf{HOD}$ does not imply $\kappa$ is $\omega$-strongly measurable in $\mathsf{HOD}$.

\begin{definition}
(Woodin, \cite[Definition 132]{Woodin 2})\quad Suppose $N$ is a proper class inner model of $\mathsf{V}$ and $\delta$ is a supercompact cardinal. Then $\delta$ is $N$-supercompact if for all $\lambda>\delta$, there exists an elementary embedding $j: \mathsf{V}\rightarrow M$ such that $crit(j)=\delta, j(\delta)>\lambda, M^{\mathsf{V}_{\lambda}}\subseteq M$ and $j(N\cap \mathsf{V}_{\delta})\cap \mathsf{V}_{\lambda}=N\cap \mathsf{V}_{\lambda}$.\footnote{The notion of $N$-supercompactness is a generalization of supercompactness. $\kappa$ is supercompact does not imply that $\kappa$ is $\mathsf{HOD}$-supercompact.}
\end{definition}

\begin{theorem}\label{extendible imply sc}
(Woodin, \cite[Lemma 188]{Woodin 2})\quad Suppose that $\delta$ is an extendible cardinal. Then $\delta$ is $\mathsf{HOD}$-supercompact.
\end{theorem}

Note that if $\delta$ is extendible, then $\delta$ is a limit of $\mathsf{HOD}$-supercompact cardinals.

\begin{definition}
(Woodin, \cite[Definition 15]{Woodin 1})\quad Suppose $N$ is a transitive class, $\mathsf{Ord}\subseteq N$ and $N\models \mathsf{ZFC}$. $N$ is a weak extender model for $\delta$ supercompact if for every $\gamma>\delta$ there exists a normal fine $\delta$-complete measure $\mathcal{U}$ on $P_{\delta}(\gamma)$ such that $N\cap P_{\delta}(\gamma)\in \mathcal{U}$ and $\mathcal{U}\cap N\in N$.
\end{definition}

The notion of ``weak extender model for $\delta$ supercompact" is very important in the study of Inner Model Theory for one supercompact cardinal. Woodin speculates that the extension to the level of one supercompact cardinal should yield as a theorem that if $\delta$ is supercompact then there exists $N\subseteq \mathsf{HOD}$ such that $N$ is a weak extender model for $\delta$ supercompact(c.f\cite{The HOD Dichotomy}).

\begin{theorem}\label{property of weak extender model}
(Woodin, \cite[Theorem 138]{Woodin 2})\quad Suppose $N$ is a weak extender model for $\delta$ supercompact and $\gamma>\delta$ is a singular cardinal. Then $\gamma$ is singular in $N$ and $\gamma^{+}=(\gamma^{+})^{N}$.
\end{theorem}

\begin{theorem}\label{Magidor supercompact}
(Magidor, \cite[Theorem 22.10]{Set theory})\quad $\delta$ is supercompact if and only if for every $\kappa>\delta$, there exist $\alpha<\delta$ and an elementary embedding $j: \mathsf{V}_{\alpha}\rightarrow \mathsf{V}_{\kappa}$ with critical point $\bar{\delta}$ such that $j(\bar{\delta})=\delta$.
\end{theorem}

The following theorem is a generalization of Magidor's characterization of supercompactness and an alternative formulation of ``weak extender model for $\delta$ supercompact" in terms of suitable elementary embeddings.

\begin{theorem}\label{characteration of extender model}
(Woodin, \cite[Theorem 21]{Woodin 1})\quad Let $N$ be a proper class inner model of $\mathsf{ZFC}$. Then the following are equivalent:\footnote{Theorem \ref{characteration of extender model}  is a reformulation of \cite[Theorem 21]{Woodin 1} in terms of Magidor's characterization of supercompactness.}
\begin{enumerate}[(1)]
  \item $N$ is a weak extender model for $\delta$ supercompact.
  \item For every $\kappa>\delta$, there exist
  $\alpha<\delta$ and an elementary embedding $j: \mathsf{V}_{\alpha+1}\rightarrow \mathsf{V}_{\kappa+1}$ such that:
\begin{enumerate}[(a)]
  \item $crit(j)=\bar{\delta}$ and $j(\bar{\delta})=\delta$;
  \item $j\upharpoonright (N\cap \mathsf{V}_{\alpha})\in N$ and $j(N\cap \mathsf{V}_{\alpha})=N\cap \mathsf{V}_{\kappa}$.
\end{enumerate}
\end{enumerate}
\end{theorem}

As a corollary of Theorem \ref{characteration of extender model}, we have the following universality theorem for weak extender model for $\delta$ supercompact.

\begin{theorem}\label{general universality theorem}
(Woodin, \cite[Theorem 144]{Woodin 2})\quad Suppose $N$ is a weak extender model for $\delta$ supercompact and $\gamma>\delta$ is a cardinal of $N$. If $j:(H_{\gamma^{+}})^{N}\rightarrow (H_{j(\gamma)^{+}})^{N}$ is an elementary embedding with $crit(j)\geq\delta$. Then $j \in N$.
\end{theorem}

\begin{theorem}\label{big thm from woodin first}
(Woodin, \cite[Theorem 193]{Woodin 2})\quad Suppose the $\mathsf{HOD}$ Hypothesis holds and $\delta$ is $\mathsf{HOD}$-supercompact. Then $\mathsf{HOD}$ is a weak extender model for $\delta$ supercompact.\footnote{Theorem \ref{big thm from woodin first} is a reformulation of \cite[Theorem 193]{Woodin 2}.}
\end{theorem}

From \cite[Lemma 136]{Woodin 2}, if $N$ is a weak extender model for $\delta$ supercompact, then $\delta$ is $N$-supercompact. So if the $\mathsf{HOD}$ Hypothesis holds, then ``$\mathsf{HOD}$ is a weak extender model for $\delta$ supercompact" is equivalent to $\delta$ is $\mathsf{HOD}$-supercompact.

From Theorem \ref{big thm from woodin first}, Theorem \ref{characteration of extender model} and Theorem \ref{Magidor supercompact}, if the $\mathsf{HOD}$ Hypothesis holds and $\kappa$ is extendible (or $\kappa$ is $\mathsf{HOD}$-supercompact) in $V$, then $\kappa$ is supercompact in $\mathsf{HOD}$.

The following remarkable Universality Theorem follows from Theorem \ref{big thm from woodin first} and Theorem \ref{general universality theorem}.

\begin{theorem}\label{big thm from woodin second}
(Woodin, Global Universality Theorem, (\cite[Theorem 201]{Woodin 2})) \quad Suppose the $\mathsf{HOD}$ Hypothesis holds and $\delta$ is $\mathsf{HOD}$-supercompact. If $j: \mathsf{HOD}\cap \mathsf{V}_{\gamma+1}\rightarrow M\subseteq \mathsf{HOD}\cap \mathsf{V}_{j(\gamma)+1}$ is an elementary embedding with $crit(j)\geq\delta$. Then $j \in \mathsf{HOD}$.
\end{theorem}

As a corollary of Theorem \ref{big thm from woodin second}, if the $\mathsf{HOD}$ Hypothesis holds and $\delta$ is $\mathsf{HOD}$-supercompact, then there is no non-trivial elementary embedding $j: \mathsf{HOD}\rightarrow \mathsf{HOD}$ such that $\delta\leq crit(j)$.

The following  definition of super-$\mathsf{HOD}$ cardinal is isolated from the proof of Theorem \ref{big thm from woodin first} in \cite[Theorem 193]{Woodin 2}. From Theorem \ref{definition equivalence}, Definition \ref{sup HOD cardinal} provides a different and equivalent definition of $\mathsf{HOD}$-supercompact cardinal.

\begin{definition}\label{sup HOD cardinal}
Define that $\kappa$ is a super-$\mathsf{HOD}$ cardinal if for any $\lambda>\kappa$ and any $a\in \mathsf{V}_{\lambda}$, there exist $j: \mathsf{V}_{\lambda_0+\omega}\rightarrow \mathsf{V}_{\lambda+\omega}, a_0\in \mathsf{V}_{\lambda_0}$ and $\kappa_0<\lambda_0<\kappa$ such that $crit(j)=\kappa_0, j(\kappa_0)=\kappa, j(a_0)=a$ and $j(\mathsf{HOD}\cap \mathsf{V}_{\lambda_0})=\mathsf{HOD}\cap \mathsf{V}_{\lambda}$.
\end{definition}

\begin{lemma}\label{key equivalence}
Suppose $\lambda>\kappa$ are uncountable regular cardinals, $|\mathsf{V}_{\lambda}|=\lambda$ and $\mathsf{HOD}\cap \mathsf{V}_{\lambda} = (\mathsf{HOD})^{\mathsf{V}_{\lambda}}$. Then the following two statements are equivalent:
\begin{enumerate}[(1)]
  \item There exists an elementary embedding $j: \mathsf{V}\rightarrow M$ such that $crit(j)=\kappa, j(\kappa)>\lambda, M^{\mathsf{V}_{\lambda}}\subseteq M$ and $j(\mathsf{HOD}\cap \mathsf{V}_{\kappa})\cap \mathsf{V}_{\lambda}=\mathsf{HOD}\cap \mathsf{V}_{\lambda}$.
  \item  There exists a normal fine $\kappa$-complete ultrafilter $U$ on $P_{\kappa}(\mathsf{V}_{\lambda})$ such that $Z \in U$ where $Z=\{X \prec \mathsf{V}_{\lambda}$: the transitive collapse of $X$ is $\mathsf{V}_{\theta}$ for some $\theta$  such that $\mathsf{HOD}\cap \mathsf{V}_{\theta} = (\mathsf{HOD})^{\mathsf{V}_{\theta}}$\}.
\end{enumerate}
\end{lemma}
\begin{proof}\label{}
It suffices to check that $Z \in U$  iff $j(\mathsf{HOD}\cap \mathsf{V}_{\kappa})\cap \mathsf{V}_{\lambda}=\mathsf{HOD}\cap \mathsf{V}_{\lambda}$. Note that $Z \in U$  iff $\{j(x): x\in \mathsf{V}_{\lambda}\}\in j(Z)$. Note that $j(Z)=\{X \prec \mathsf{V}_{j(\lambda)}$: the transitive collapse of $X$ is $\mathsf{V}_{\theta}$ for some $\theta$  such that $j(\mathsf{HOD}\cap \mathsf{V}_{\kappa})\cap \mathsf{V}_{\theta} = (\mathsf{HOD})^{\mathsf{V}_{\theta}}$\}. Since the transitive collapse of $\{j(x): x\in \mathsf{V}_{\lambda}\}$ is $\mathsf{V}_{\lambda}$ and $\mathsf{HOD}\cap \mathsf{V}_{\lambda} = (\mathsf{HOD})^{\mathsf{V}_{\lambda}}$, we have $\{j(x): x\in \mathsf{V}_{\lambda}\}\in j(Z)$ iff $j(\mathsf{HOD}\cap \mathsf{V}_{\kappa})\cap \mathsf{V}_{\lambda} = (\mathsf{HOD})^{\mathsf{V}_{\lambda}}=\mathsf{HOD}\cap \mathsf{V}_{\lambda}$.
\end{proof}

\begin{theorem}\label{definition equivalence}
The following three statements are equivalent:
\begin{enumerate}[(1)]
  \item $\kappa$ is $\mathsf{HOD}$-supercompact.
  \item $\kappa$ is a super-$\mathsf{HOD}$ cardinal.
  \item  For any $\lambda>\kappa$, there exists a normal fine $\kappa$-complete ultrafilter $U$ on $P_{\kappa}(\mathsf{V}_{\lambda})$ such that $Z \in U$ where $Z=\{X \prec \mathsf{V}_{\lambda}$: the transitive collapse of $X$ is $\mathsf{V}_{\theta}$ for some $\theta$  such that $\mathsf{HOD}\cap \mathsf{V}_{\theta} = (\mathsf{HOD})^{\mathsf{V}_{\theta}}$\}.\footnote{The isolation of this statement as the equivalence of $\mathsf{HOD}$-supercompact cardinal is due to Woodin from \cite{Woodin 5}.}
\end{enumerate}
\end{theorem}
\begin{proof}\label{}
$(1)\Rightarrow (2)$: From \cite[Lemma 133]{Woodin 2}, if $\delta$ is $\mathsf{HOD}$-supercompact, then $\delta$ is super-$\mathsf{HOD}$ cardinal.

$(2)\Rightarrow (3)$: Suppose $\kappa$ is super-$\mathsf{HOD}$ cardinal, $\lambda > \kappa$, and $\mathsf{V}_{\lambda}$ is a $\Sigma_2$ elementary substructure of $\mathsf{V}$ such that $|\mathsf{V}_{\kappa}|=\kappa$ and $\mathsf{HOD}\cap \mathsf{V}_{\kappa}=(\mathsf{HOD})^{\mathsf{V}_{\kappa}}$. We show that there exists a normal fine $\kappa$-complete ultrafilter $U$ on $P_{\kappa}(\mathsf{V}_{\lambda})$ such that $Z \in U$ where $Z=\{X \prec \mathsf{V}_{\lambda}$: the transitive collapse of $X$ is $\mathsf{V}_{\theta}$ for some $\theta$  such that $\mathsf{HOD}\cap \mathsf{V}_{\theta} = (\mathsf{HOD})^{\mathsf{V}_{\theta}}$\}.

Since $\kappa$ is super-$\mathsf{HOD}$, there exists $\overline{\kappa}< \overline{\lambda} < \kappa$ and an elementary embedding $\pi: \mathsf{V}_{\overline{\lambda}+\omega} \to \mathsf{V}_{\lambda+\omega}$ such that $crit(\pi) = \overline{\kappa}, \pi(\overline{\lambda}) = \lambda$, and $\pi(\mathsf{HOD}\cap \mathsf{V}_{\overline{\lambda}}) = \mathsf{HOD}\cap \mathsf{V}_{\lambda}$. Let $\overline{U}$ be the $\overline{\kappa}$-complete normal fine ultrafilter on $P_{\overline{\kappa}}(\mathsf{V}_{\overline{\lambda}})$ given by $\pi$. Thus $\overline{U} \in \mathsf{V}_{\overline{\lambda}+\omega}$. Then $\pi(\overline{U})$ is a $\kappa$-complete normal fine ultrafilter on $P_{\kappa}(\mathsf{V}_{\lambda})$.

It suffices to show that $Z\in  \pi(\overline{U})$. Let $\pi(\overline{Z})=Z$ and $\sigma_{\pi}=\lbrace\pi(a)$ : $a \in \mathsf{V}_{\overline{\lambda}}\rbrace$. Since $\mathsf{HOD}\cap \mathsf{V}_{\lambda} = (\mathsf{HOD})^{\mathsf{V}_{\lambda}}$ and  $\pi(\mathsf{HOD}\cap \mathsf{V}_{\overline{\lambda}}) = \mathsf{HOD}\cap \mathsf{V}_{\lambda}$, we have $\mathsf{HOD}\cap \mathsf{V}_{\overline{\lambda}} = (\mathsf{HOD})^{\mathsf{V}_{\overline{\lambda}}}$.  Thus $\sigma_{\pi} \in Z$. Note that $\overline{Z}\in \overline{U}$ iff $\sigma_{\pi}\in \pi(\overline{Z})$. So $\overline{Z}\in \overline{U}$ and hence $Z\in  \pi(\overline{U})$.

$(3)\Rightarrow (1)$: Follows from Lemma \ref{key equivalence} since we can only consider $\lambda>\kappa$ such that $|\mathsf{V}_{\lambda}|=\lambda$ and $\mathsf{HOD}\cap \mathsf{V}_{\lambda} = (\mathsf{HOD})^{\mathsf{V}_{\lambda}}$.
\end{proof}

\begin{definition}\label{}
For regular cardinals $\delta<\kappa$, we say $(\delta, \kappa)$ is a $\mathsf{HOD}$-partition pair if there exists a  partition $\langle S_{\alpha}\mid \alpha<\delta\rangle\in \mathsf{HOD}$ of $\{\alpha<\kappa\mid cf(\alpha)=\omega\}$ into pairwise disjoint stationary sets.
\end{definition}
If $\mathsf{V}=\mathsf{HOD}$, then for any regular cardinals $\delta<\kappa$, $(\delta, \kappa)$ is a $\mathsf{HOD}$-partition pair. Note that the $\mathsf{HOD}$ Hypothesis implies that for any $\delta$ there is regular cardinal $\kappa>\delta$ such that $(\delta, \kappa)$ is a $\mathsf{HOD}$-partition pair.

\begin{theorem}\label{HOD partion pair}
(Woodin, \cite[Theorem 195]{Woodin 2})\quad If $\delta$ is $\mathsf{HOD}$-supercompact and for any $\gamma>\delta$ there is regular cardinal $\lambda>\gamma$ such that $(\gamma, \lambda)$ is a $\mathsf{HOD}$-partition pair. Then $\mathsf{HOD}$ is a weak extender model for $\delta$ supercompact.
\end{theorem}

The following Theorem \ref{big thm from woodin in SEM} and Theorem \ref{equi be HC and MC} are a reformulation and summary of Woodin's results in \cite{Woodin 2} (eg. \cite[Theorem 212]{Woodin 2}, \cite[Theorem 19]{Woodin 1}, etc.).

\begin{theorem}\label{big thm from woodin in SEM}
(Woodin, \cite{Woodin 2})\quad Suppose $\delta$ is $\mathsf{HOD}$-supercompact.\footnote{Theorem 212 in \cite{Woodin 2} assumes that $\delta$ is extendible. In fact it suffices to assume that $\delta$ is $\mathsf{HOD}$-supercompact.} Then the following are equivalent:
\begin{enumerate}[(1)]
  \item The $\mathsf{HOD}$ Hypothesis.
  \item $\mathsf{HOD}$ is a weak extender model for $\delta$ supercompact.
  \item There exists a weak extender model $N$ for $\delta$ supercompact
  such that $N \subseteq \mathsf{HOD}$.
  \item Every singular cardinal $\gamma>\delta$ is singular in $\mathsf{HOD}$ and $\gamma^{+}=(\gamma^{+})^{\mathsf{HOD}}$.
  \item There is a proper class of regular cardinals that are not measurable in $\mathsf{HOD}$.
  \item For any $\gamma>\delta$ there is regular cardinal $\lambda>\gamma$ such that $(\gamma, \lambda)$ is a $\mathsf{HOD}$-partition pair.
\end{enumerate}
\end{theorem}
\begin{proof}
By Theorem \ref{big thm from woodin first}, $(1)\Rightarrow (2)$. By Theorem \ref{property of weak extender model}, $(2)\Rightarrow (4)$. By Theorem \ref{strong mc imply mc}, $(5)\Rightarrow (1)$. By Theorem \ref{HOD partion pair}, $(6)\Rightarrow (2)$. It is a theorem in $\mathsf{ZFC}$ that $(2)\Rightarrow (3)$, $(1)\Rightarrow (6)$, $(4)\Rightarrow (1)$ and $(4)\Rightarrow (5)$. By Theorem \ref{characteration of extender model} and Theorem \ref{Magidor supercompact}, $\delta$ is supercompact in $\mathsf{HOD}$ if and only if $\mathsf{HOD}$ is a weak extender model for $\delta$ supercompact.

$(4)\Rightarrow (1)$: if (4) holds, then $\lbrace\gamma^{+}: \gamma>\delta$ is a singular cardinal$\rbrace$ is a proper class of regular cardinals which are not $\omega$-strongly measurable in $\mathsf{HOD}$. By the similar argument, we have  $(4)\Rightarrow (5)$.

Finally, it suffices to show that $(3)\Rightarrow (1)$. By Theorem \ref{property of weak extender model}, $(3)\Rightarrow (4)$. Since $(4)\Rightarrow (1)$, we have $(3)\Rightarrow (1)$.
\end{proof}

\begin{theorem}\label{equi be HC and MC}
(Woodin, \cite{Woodin 2})\quad Suppose $\delta$ is extendible. Then the following are equivalent:
\begin{enumerate}[(1)]
  \item The $\mathsf{HOD}$ Hypothesis.
  \item There exists a regular cardinal $\kappa\geqslant\delta$ such that $\kappa$ is not measurable in $\mathsf{HOD}$.
   \item There exists a regular cardinal $\kappa\geqslant\delta$ such that $(\delta,\kappa)$ is a $\mathsf{HOD}$-partition pair.
   \item For any cardinal $\kappa$, if $\kappa$ is $\mathsf{HOD}$-supercompact, then $\mathsf{HOD}$ is a weak extender model for $\kappa$-supercompact.
   \item There exists a regular cardinal $\kappa\geqslant\delta$ such that $\kappa$ is not $\omega$-strongly measurable in $\mathsf{HOD}$.
\end{enumerate}
\end{theorem}
\begin{proof}
By Theorem \ref{big thm from woodin first} and Theorem \ref{property of weak extender model}, $(1)\Rightarrow (2)$. It is a theorem in $\mathsf{ZFC}$ that $(1)\Rightarrow (3)$.

By Theorem \ref{extendible imply sc} and Theorem \ref{big thm from woodin in SEM}, $(4)\Rightarrow (1)$. $(2)\Rightarrow (1)$: Let $I$ be the set of regular cardinals $\gamma$ such that there exists $\eta>\gamma$ such that $\mathsf{V}_{\eta}\models \mathsf{ZFC}$ and $\mathsf{V}_{\eta}\models \gamma$ is not $\omega$-strongly measurable in $\mathsf{HOD}$. Note that if $\gamma\in I$, then $\gamma$ is not $\omega$-strongly measurable in $\mathsf{HOD}$. Since $\kappa$ is not measurable in $\mathsf{HOD}, \kappa$ is not $\omega$-strongly measurable in $\mathsf{HOD}$ in $\mathsf{V}_{\eta}$ for sufficiently large $\eta$ and hence $\kappa\in I$. Let $\eta$ be the witness of $\kappa\in I$. Since $\delta$ is extendible, for any $\alpha$, there exists $j: \mathsf{V}_{\eta+1}\rightarrow \mathsf{V}_{j(\eta)+1}$ such that $crit(j)=\delta$ and $j(\delta)>\alpha$. Then $j(\eta)$ witnesses that $j(\kappa)\in I$ and $j(\kappa)>\alpha$.

$(3)\Rightarrow (4)$: Suppose  there exists a regular cardinal $\kappa>\delta$ such that $(\delta,\kappa)$ is a $\mathsf{HOD}$-partition pair and $\kappa$ is $\mathsf{HOD}$-supercompact. Let $\theta>\kappa$ be large enough such that $\mathsf{HOD}\bigcap 2^{\kappa}=\mathsf{HOD}^{\mathsf{V}_{\theta}}\bigcap 2^{\kappa}$. Let $j: \mathsf{V}_{\theta+1}\rightarrow \mathsf{V}_{j(\theta)+1}$ be an elementary embedding such that $crit(j)=\delta$ and $j(\delta)>\kappa$. Let $\varphi(\delta)$ denote the statement: for any regular $\lambda<\delta$ there exists regular $\gamma>\lambda$ such that $(\lambda,\gamma)$ is a $\mathsf{HOD}$-partition pair. Note that $\mathsf{V}\models \varphi(\delta)$. Since $\mathsf{HOD}\bigcap 2^{\kappa}=\mathsf{HOD}^{\mathsf{V}_{\theta}}\bigcap 2^{\kappa}, \mathsf{V}_{\theta}\models \varphi(\delta)$. By elementarity of $j$, since $\mathsf{HOD}^{\mathsf{V}_{j(\theta)}}\subseteq \mathsf{HOD}$, for any $\lambda<j(\delta)$ there exists $\gamma>\lambda$ such that $(\lambda,\gamma)$ is a $\mathsf{HOD}$-partition pair. Since $j$ can be chosen with $j(\delta)$ arbitrarily large, it follows that for any $\lambda$ there exists $\gamma>\lambda$ such that $(\lambda,\gamma)$ is a $\mathsf{HOD}$-partition pair. By Theorem \ref{HOD partion pair}, $\mathsf{HOD}$ is a weak extender model for $\kappa$-supercompact.
\end{proof}

In the following, we discuss some basic facts about forcing with respect to the $\mathsf{HOD}$ Hypothesis. It is not hard to force statements listed in Theorem \ref{big thm from woodin in SEM} and Theorem \ref{equi be HC and MC}. Suppose  $\kappa$ is a supercompact cardinal and $\phi$ is any statement listed in Theorem \ref{big thm from woodin in SEM} and Theorem \ref{equi be HC and MC}. Then one can force that $\mathsf{V} \neq \mathsf{HOD}$, $\kappa$ is supercompact and $\phi$ holds as follows: First force to make $\kappa$ indestructible with the appropriate preparatory forcing, then force $\mathsf{V} = \mathsf{HOD}$ and finally add a Cohen real; In the final model,  $\mathsf{V} \neq \mathsf{HOD}, \kappa$ is supercompact and $\phi$ holds.

It is not hard to force the $\mathsf{HOD}$ Hypothesis since $\mathsf{V}=\mathsf{HOD}$ implies the $\mathsf{HOD}$ Hypothesis. It is a folklore that relative to $\mathsf{ZFC}$, we can force $\mathsf{V}=\mathsf{HOD}$ by a proper class forcing notion. For nearly any known large cardinal notion $\phi$, relative to ``$\mathsf{ZFC}\,+\, \phi$" we can force that $\mathsf{V}=\mathsf{HOD}$ and $\phi$ holds. There is a simple way to force that $\mathsf{V} \neq \mathsf{HOD}$ and the $\mathsf{HOD}$ Hypothesis holds: First force $\mathsf{V} =\mathsf{HOD}$ and then add a Cohen real.

\begin{lemma}\label{key lemma on forcing}
Suppose $\delta$ is $\mathsf{HOD}$-supercompact, $\mathbb{P}\in \mathsf{V}_{\delta}$ and $G$ is $\mathbb{P}$-generic over $\mathsf{V}$. If $\mathbb{P}$ is weakly homogeneous and ordinal definable, then $\mathsf{V}\models$ The $\mathsf{HOD}$ Hypothesis if and only if $\mathsf{V}[G]\models$ The $\mathsf{HOD}$ Hypothesis.
\end{lemma}
\begin{proof}
From Theorem \ref{big thm from woodin first}, the $\mathsf{HOD}$ Hypothesis is equivalent to the statement: every singular cardinal $\gamma>\delta$ is singular in $\mathsf{HOD}$ and $\gamma^{+}=(\gamma^{+})^{\mathsf{HOD}}$. Suppose $\mathsf{V}\models$ The $\mathsf{HOD}$ Hypothesis. We show that $\mathsf{V}[G]\models$ The $\mathsf{HOD}$ Hypothesis. Suppose in $\mathsf{V}[G], \gamma>\delta$ is singular. Then in $\mathsf{V}, \gamma>\delta$ is singular. Since $\mathsf{V}\models$ The $\mathsf{HOD}$ Hypothesis, $\gamma$ is singular in $\mathsf{HOD}^{\mathsf{V}}$. By Proposition \ref{key fact on forcing}, $\mathsf{HOD}^{\mathsf{V}}$ is a $\delta$-c.c. generic extension of $\mathsf{HOD}^{\mathsf{V}[G]}$. Then $\gamma$ is singular in $\mathsf{HOD}^{\mathsf{V}[G]}$. Note that $(\gamma^{+})^{\mathsf{V}[G]}= \gamma^{+}= (\gamma^{+})^{\mathsf{HOD}^{\mathsf{V}}}= (\gamma^{+})^{\mathsf{HOD}^{\mathsf{V}[G]}}$. So in $\mathsf{V}[G]$, if $\gamma>\delta$ is singular, then $\gamma$ is singular in $\mathsf{HOD}$ and $\gamma^{+}=(\gamma^{+})^{\mathsf{HOD}}$. By a similar argument, we can show that if $\mathsf{V}[G]\models$ The $\mathsf{HOD}$ Hypothesis, then $\mathsf{V}\models$ The $\mathsf{HOD}$ Hypothesis.
\end{proof}

\begin{proposition}
(\cite[Corollary 8]{Woodin 1})\quad Suppose $\delta$ is $\mathsf{HOD}$-supercompact, $G$ is $\mathbb{P}$-generic over $\mathsf{V}$ and $\mathbb{P}\in \mathsf{V}_{\delta}$. Then $\mathsf{V}\models$ The $\mathsf{HOD}$ Hypothesis if and only if $\mathsf{V}[G]\models$ The $\mathsf{HOD}$ Hypothesis.
\end{proposition}
\begin{proof}
Take $\kappa<\delta$ be an inaccessible cardinal such that $\mathbb{P}\in \mathsf{V}_{\kappa}$.
Let $I$ be a $\mathsf{V}[G]$-generic filter on $Coll(\omega, \kappa)$ and $J$ be a $\mathsf{V}$-generic filter on $Coll(\omega, \kappa)$ such that $\mathsf{V}[G][I]=\mathsf{V}[J]$. Since $Coll(\omega, \kappa)$ is ordinal definable and weakly homogeneous, by Lemma \ref{key lemma on forcing}, the $\mathsf{HOD}$ Hypothesis is absolute between $\mathsf{V}[G]$ and $\mathsf{V}[G][I]$, as well as between $\mathsf{V}$ and $\mathsf{V}[J]$. So the $\mathsf{HOD}$ Hypothesis is absolute between $\mathsf{V}$ and $\mathsf{V}[G]$.
\end{proof}

\begin{corollary}\label{class absolute for strong hod}\footnote{This corollary strengthens theorem 214 in \cite{Woodin 2}.}
\begin{enumerate}[(1)]
  \item  Suppose $\delta$ is $\mathsf{HOD}$-supercompact. Then $\mathsf{V}\models$ The $\mathsf{HOD}$ Hypothesis iff for any partial order $\mathbb{P}\in \mathsf{V}_{\delta}, V^{\mathbb{P}}\models$ The $\mathsf{HOD}$ Hypothesis.
  \item  If there exists a proper class of $\mathsf{HOD}$-supercompact  cardinals, then $\mathsf{V}\models$ The $\mathsf{HOD}$ Hypothesis iff for any partial order $\mathbb{P}, \mathsf{V}^{\mathbb{P}}\models$ The $\mathsf{HOD}$ Hypothesis.
\end{enumerate}
\end{corollary}

\begin{definition}\label{}
The Strong $\mathsf{HOD}$ Hypothesis denotes the statement: there is a proper class of regular cardinals which are not measurable in $\mathsf{HOD}$.
\end{definition}

By Theorem \ref{strong mc imply mc}, the Strong $\mathsf{HOD}$ Hypothesis implies the $\mathsf{HOD}$ Hypothesis. By Theorem \ref{big thm from woodin in SEM}, if there exists an $\mathsf{HOD}$-supercompact cardinal, then the $\mathsf{HOD}$ Hypothesis is equivalent to the Strong $\mathsf{HOD}$ Hypothesis. From Corollary \ref{class absolute for strong hod}, if $\delta$ is $\mathsf{HOD}$-supercompact, then $\mathsf{V}\models$ The Strong $\mathsf{HOD}$ Hypothesis iff for any partial order $\mathbb{P}\in \mathsf{V}_{\delta}, \mathsf{V}^{\mathbb{P}}\models$ The Strong $\mathsf{HOD}$ Hypothesis. The difficulty in forcing the failure of the Strong $\mathsf{HOD}$ Hypothesis comes from the difficulty in making the successors of singular cardinals measurable in  $\mathsf{HOD}$.

\section{The $\mathsf{HOD}$ Hypothesis and a supercompact cardinal}

In \cite{joint work with Joel}, very large cardinals such as supercompact cardinals in $\mathsf{V}$ are forced not to exhibit their large cardinal properties in $\mathsf{HOD}$: they can be very small (not even weakly compact) in $\mathsf{HOD}$. A reasonable natural question would be how far this can be taken, that is whether there exists a supercompact cardinal in $\mathsf{V}$ which is not only not even weakly compact in $\mathsf{HOD}$ but also has no other cardinals in $\mathsf{HOD}$ which exhibit large cardinal behavior. In the following, we prove in Theorem \ref{main result of the paper} that under the $\mathsf{HOD}$ Hypothesis the answer is no: if $\kappa$ is supercompact and  the $\mathsf{HOD}$ Hypothesis holds, then there is a proper class of regular cardinals below $\kappa$, which are measurable in $\mathsf{HOD}$.

The main idea of Theorem \ref{main result of the paper} is as follows. Suppose $\kappa$ is supercompact and the $\mathsf{HOD}$ Hypothesis holds. Take $\alpha<\kappa$ and $\lambda>\kappa$ such that $\lambda$ is a limit of regular cardinals which are not $\omega$-strongly measurable in $\mathsf{HOD}$ and $\mathsf{HOD}\cap \mathsf{V}_{\lambda}=\mathsf{HOD}^{\mathsf{V}_{\lambda}}$. To find a measurable cardinal between $\alpha$ and $\kappa$ in $\mathsf{HOD}$, we need find elementary embeddings $\pi_1: \mathsf{V}_{\lambda_1+1} \to \mathsf{V}_{\lambda_2+1}$ and  $\pi_2: \mathsf{V}_{\lambda_2+1} \to \mathsf{V}_{\lambda+1}$ such that $crit(\pi_1)=\kappa_1, \alpha<\kappa_1<\kappa$ and $\pi_3(\mathsf{HOD}\cap \mathsf{V}_{\lambda_1}) \subseteq\mathsf{HOD}$ where $\pi_3=\pi_2\circ\pi_1$, which we will get in Theorem \ref{equivalence of supercompact}. We want to show that $\kappa_1$ is measurable in $\mathsf{HOD}$. To do this, it suffices to show that any $\overline{\gamma}<\lambda_1, \pi_3\upharpoonright (\mathsf{HOD}\cap \mathsf{V}_{\overline{\gamma}})\in \mathsf{HOD}$. Suppose $\pi_3(\overline{\gamma})=\gamma$. Take  $\delta>|\mathsf{V}_{\gamma+\omega+1}|$ such that $\delta<\lambda$ and $\delta$ is not $\omega$-strongly measurable in $\mathsf{HOD}$. By the $\mathsf{HOD}$ Hypothesis, there exists $\langle S_{\alpha}\mid\alpha< |\mathsf{V}_{\gamma+\omega}|\rangle\in \mathsf{HOD}$ which is  a  partition of $S^{\delta}_{\omega}$ into  stationary sets in $\delta$. Then applying Lemma \ref{isolated lemma} to $\pi_3$ and $\langle S_{\alpha}\mid\alpha< |\mathsf{V}_{\gamma+\omega}|\rangle$, we have $\tau_{\eta_0}=\pi_3\upharpoonright |\mathsf{V}_{\overline{\gamma}+\omega}|$. From $\langle S_{\alpha}\mid\alpha< |\mathsf{V}_{\gamma+\omega}|\rangle\in \mathsf{HOD}$, we can show that  $\pi_3\upharpoonright |\mathsf{V}_{\overline{\gamma}+\omega}|\in \mathsf{HOD}$. Finally, from $\pi_3(\mathsf{HOD}\cap \mathsf{V}_{\lambda_1}) \subseteq\mathsf{HOD}$ and $\pi_3\upharpoonright |\mathsf{V}_{\overline{\gamma}+\omega}|\in \mathsf{HOD}$, by a standard argument, we can show that $\pi_3\upharpoonright (\mathsf{HOD}\cap \mathsf{V}_{\overline{\gamma}})\in \mathsf{HOD}$.

The following Theorem \ref{equivalence of supercompact} gives a new formulation of supercompactness which is important in the proof of our main result Theorem \ref{main result of the paper}. Compared to Magidor's characterization of supercompactness in Theorem \ref{Magidor supercompact}, the new component of this formulation is the coherence condition for $\pi_1$ in Theorem \ref{equivalence of supercompact}(3). From \cite{Woodin 5}, Woodin also proved Theorem \ref{equivalence of supercompact}.

The idea behind Theorem \ref{equivalence of supercompact} is as follows. Let $j_0: \mathsf{V}\rightarrow M_0$ be the witness embedding for $\kappa$-supercompactness such that $M_0$ is closed under $\mathsf{V}_{\lambda+1}$-sequences.  Then $j_0(j_0): M_0\rightarrow M_1$. Take $j=j_0(j_0)\circ j_0$.   Then $j: \mathsf{V}\rightarrow M_1$. Then we can get the coherence of the intermediate embeddings in $\mathsf{V}$ via showing in $M_1$ the coherence of the intermediate embeddings $\pi_1=j_0\upharpoonright \mathsf{V}_{\lambda+1}$ and $\pi_2=j_0(\pi_1)=j_0(j_0)\upharpoonright j_0(\mathsf{V}_{\lambda+1})$.

\begin{theorem}\label{equivalence of supercompact}
$\kappa$ is supercompact  if and only if for all $\lambda > \kappa$, any $\alpha<\kappa$ and for all $N \subseteq \mathsf{V}_{\kappa}$, there exist  $\kappa_1 < \lambda_1 < \kappa_2 < \lambda_2 < \kappa$, and elementary embeddings $\pi_1: \mathsf{V}_{\lambda_1+1} \to \mathsf{V}_{\lambda_2+1}$ and  $\pi_2: \mathsf{V}_{\lambda_2+1} \to \mathsf{V}_{\lambda+1}$ such that
\begin{enumerate}[(1)]
  \item $\alpha<\kappa_1, crit(\pi_1)=\kappa_1$ and $crit(\pi_2)=\kappa_2$;
  \item $\pi_2(\kappa_2) = \kappa$ and $\pi_1(\kappa_1) = \kappa_2$; and
  \item $\pi_1(N\cap \mathsf{V}_{\lambda_1}) = N \cap \mathsf{V}_{\lambda_2}$ and $\pi_2(N\cap \mathsf{V}_{\lambda_2}) = N \cap \mathsf{V}_{\lambda}$.
\end{enumerate}
\end{theorem}
\begin{proof}
Fix $\lambda > \kappa, \alpha<\kappa$ and $N \subseteq \mathsf{V}_{\kappa}$. Take $j_0: \mathsf{V}\rightarrow M_0$ such that $crit(j_0)=\kappa$ and $M_0$ is closed under $\mathsf{V}_{\lambda+1}$-sequences. Then $j_0(j_0): M_0\rightarrow M_1$ and $M_1$ is closed under $j_0(\mathsf{V}_{\lambda+1})$-sequences in $M_0$.  Let $j=j_0(j_0)\circ j_0$. Then $j: V\rightarrow M_1$.
It suffices to show in $M_1$ that there exist $\kappa_1 < \lambda_1 < \kappa_2 < \lambda_2 < j(\kappa)$, $\pi_1: \mathsf{V}_{\lambda_1+1} \to \mathsf{V}_{\lambda_2+1}$ and  $\pi_2: \mathsf{V}_{\lambda_2+1} \to \mathsf{V}_{j(\lambda)+1}$ such that
\begin{enumerate}[(1)]
  \item $j(\alpha)=\alpha<\kappa_1, crit(\pi_1)=\kappa_1$ and $crit(\pi_2)=\kappa_2$;
  \item $\pi_2(\kappa_2) = j(\kappa)$ and $\pi_1(\kappa_1) = \kappa_2$; and
  \item $\pi_1(j(N)\cap \mathsf{V}_{\lambda_1}) = j(N) \cap \mathsf{V}_{\lambda_2}$ and $\pi_2(j(N)\cap \mathsf{V}_{\lambda_2}) = j(N) \cap \mathsf{V}_{j(\lambda)}$.
\end{enumerate}
Let $\kappa_1=\kappa,  \lambda_1=\lambda, \kappa_2=j_0(\kappa)$, $\lambda_2=j_0(\lambda)$, $\pi_1=j_0\upharpoonright \mathsf{V}_{\lambda+1}$ and $\pi_2=j_0(\pi_1)=j_0(j_0)\upharpoonright j_0(\mathsf{V}_{\lambda+1})$.
Then $\pi_1: \mathsf{V}_{\lambda+1}\rightarrow \mathsf{V}_{j_0(\lambda)+1}$ and  $\pi_2: \mathsf{V}_{j_0(\lambda)+1}\rightarrow \mathsf{V}_{j(\lambda)+1}$. Since $M_1$ is closed under $\mathsf{V}_{\lambda+1}$-sequences in $\mathsf{V}$, $\pi_1\in M_1$. Since $M_1$ is closed under $j_0(\mathsf{V}_{\lambda+1})$-sequences in $M_0$, $\pi_2\in M_1$. It is easy to check that (1) and (2) hold. We only check (3) as follows. Since $N\subseteq \mathsf{V}_{\kappa}$ and $crit(j_0)=\kappa, j_0(N)=N$. Note that $\pi_1(j(N)\cap \mathsf{V}_{\lambda})=j_0(j_0(j_0)(j_0(N)))\cap \mathsf{V}_{j_0(\lambda)}=j_0(j_0(j_0(N)))\cap \mathsf{V}_{j_0(\lambda)}=j_0(j_0(N))\cap \mathsf{V}_{j_0(\lambda)}=j(N)\cap \mathsf{V}_{j_0(\lambda)}$. By the similar argument, we have $\pi_2(j(N)\cap \mathsf{V}_{j_0(\lambda)}) = j(N) \cap \mathsf{V}_{j(\lambda)}$.
\end{proof}

The following lemma is isolated from Woodin's proof of Theorem \ref{big thm from woodin first} in \cite[Theorem 193]{Woodin 2}. Since Lemma \ref{isolated lemma} will be used in the proof of Theorem \ref{main result of the paper}, we prove it with details here.  The technique in the proof of Lemma \ref{isolated lemma} also appears in Woodin's lemma in \cite[Theorem 11]{Kunen inconsistency}.
\begin{lemma}\label{isolated lemma}
Suppose \ $\kappa$ is an uncountable regular cardinal, \ $\gamma<\kappa$, and $\langle S_{\alpha}: \alpha<\gamma\rangle$ is a partition of \ $cof(\omega)\cap\kappa$ into stationary sets. Let \ $j$ be an elementary embedding with critical point \ $\delta$ such that \ $j(\delta)<\gamma$. Let \ $j(\overline{\gamma},\overline{\kappa})=(\gamma,\kappa)$ and \ $j(\langle \overline{S}_{\alpha}: \alpha<\overline{\gamma}\rangle)=\langle S_{\alpha}: \alpha<\gamma\rangle$. For \ $\eta<\overline{\kappa}$ such that \ $cof(\eta)>\omega$, let \ $\sigma_{\eta}=\{\alpha<\overline{\gamma}: \overline{S}_{\alpha}\cap\eta$ is stationary in \ $\eta\}$. Let \ $\langle \tau_{\eta}: \eta<\kappa, cof(\eta)>\omega\rangle=j(\langle \sigma_{\eta}: \eta<\overline{\kappa}, cof(\eta)>\omega\rangle)$. Let \ $\eta_0=\sup\{j(\varepsilon): \varepsilon<\overline{\kappa}\}$. Then \ $\tau_{\eta_0}=\{j(\alpha): \alpha<\overline{\gamma}\}$ (i.e. for $\theta<\gamma, \theta\in ran(j)$ if and only if $S_{\theta}\cap\eta_0$ is stationary in $\eta_0$).
\end{lemma}
\begin{proof}\label{}
Note that \ $\langle \overline{S}_{\alpha}: \alpha<\overline{\gamma}\rangle$
is a partition of \ $cof(\omega)\cap\overline{\kappa}$ into stationary sets. For
\ $\eta<\kappa, cof(\eta)>\omega, \tau_{\eta}=\{\alpha<\gamma: S_{\alpha}\cap\eta$ is stationary in \ $\eta\}$. Note that \ $\eta_0<\kappa$ and \ $cof(\eta_0)=\overline{\kappa}>\omega$.

It is easy to check that for any club \ $C\subseteq \eta_0$ there exists a club \ $D\subseteq \overline{\kappa}$ such that \ $\{j(\varepsilon): \varepsilon\in D, cof(\varepsilon)=\omega\}\subseteq \{\varepsilon\in C: cof(\varepsilon)=\omega\}$.

We first show that \ $\tau_{\eta_0}\subseteq\{j(\alpha): \alpha<\overline{\gamma}\}$.
Suppose \ $\beta\in \tau_{\eta_0}$. Then \ $S_{\beta}\cap\eta_0$ is stationary in \ $\eta_0$. Let \ $C$ be any club in \ $\eta_0$. Then there exists \ $\varepsilon\in C\cap S_{\beta}$. Let \ $\overline{\varepsilon}$ be the preimage of \ $\varepsilon$ under \ $j$. Note that \ $\bigcup_{\alpha<\overline{\gamma}} \overline{S}_{\alpha}=\overline{\kappa}\cap cof(\omega)$. So there exists \ $\alpha<\overline{\gamma}$ such that \ $\overline{\varepsilon}\in \overline{S}_{\alpha}$. Then \ $\varepsilon\in S_{j(\alpha)}$. Since \ $\varepsilon\in S_{j(\alpha)}\cap S_{\beta}, \beta=j(\alpha)$.

Next we show that \ $\{j(\alpha): \alpha<\overline{\gamma}\}\subseteq \tau_{\eta_0}$. Suppose there exists \ $\overline{\alpha}<\overline{\gamma}$ such that \ $j(\overline{\alpha})=\alpha$ but \ $\alpha\notin \tau_{\eta_0}$. Then there exists club \ $C_{\alpha}\subseteq\eta_0$ such that \ $C_{\alpha}\cap S_{\alpha}=\emptyset$. Then there exists a club \ $D_{\alpha}\subseteq \overline{\kappa}$ such that \ $\{j(\varepsilon): \varepsilon\in D_{\alpha}, cof(\varepsilon)=\omega\}\subseteq \{\varepsilon\in C_{\alpha}: cof(\varepsilon)=\omega\}$. Then there exists \ $\zeta\in D_{\alpha}\cap \overline{S}_{\overline{\alpha}}$ such that \ $j(\zeta)\in C_{\alpha}$. So \ $j(\zeta)\in C_{\alpha}\cap S_{\alpha}$ which leads to a contradiction.
\end{proof}

Now we prove the main result Theorem \ref{main result of the paper}. The idea of Theorem \ref{main result of the paper} comes from Theorem \ref{equivalence of supercompact} and proof of Theorem \ref{big thm from woodin first} in Theorem \cite[Theorem 193]{Woodin 2}. We first prove this main result and then give a summary of the proof in the end.

\begin{theorem}\label{main result of the paper}
Suppose $\kappa$ is supercompact and the $\mathsf{HOD}$ Hypothesis holds. Then  for each $\alpha<\kappa$,  there exists $\gamma$ such that $\alpha<\gamma<\kappa$ and $\gamma$ is measurable in $\mathsf{HOD}$.
\end{theorem}
\begin{proof}\label{}
Fix $\alpha<\kappa$. Take $\lambda>\kappa$ such that $|\mathsf{V}_{\lambda}| =\lambda$, $\lambda$ is a limit of regular cardinals which are not $\omega$-strongly measurable in $\mathsf{HOD}$ and $\mathsf{HOD}\cap \mathsf{V}_{\lambda}=\mathsf{HOD}^{\mathsf{V}_{\lambda}}$. Let $N=\mathsf{HOD}\cap \mathsf{V}_{\kappa}$. By Theorem \ref{equivalence of supercompact},  there exist  $\kappa_1 < \lambda_1 < \kappa_2 < \lambda_2 < \kappa$, and elementary embeddings $\pi_1: \mathsf{V}_{\lambda_1+1} \to \mathsf{V}_{\lambda_2+1}$ and  $\pi_2: \mathsf{V}_{\lambda_2+1} \to \mathsf{V}_{\lambda+1}$ such that $crit(\pi_1)=\kappa_1, \alpha<\kappa_1$, $\pi_1(\mathsf{HOD}\cap \mathsf{V}_{\lambda_1}) = \mathsf{HOD} \cap \mathsf{V}_{\lambda_2}$ and $\pi_2(\mathsf{HOD}\cap \mathsf{V}_{\lambda_2}) = \mathsf{HOD} \cap \mathsf{V}_{\kappa}\cap \mathsf{V}_{\lambda}=\mathsf{HOD} \cap \mathsf{V}_{\kappa}$. Let $\pi_3=\pi_2\circ\pi_1$. We want to show that $\kappa_1$ is measurable in $\mathsf{HOD}$. Since $crit(\pi_3)=\kappa_1$, it suffices to show that for any $\overline{\gamma}<\lambda_1, \pi_3\upharpoonright (\mathsf{HOD}\cap \mathsf{V}_{\overline{\gamma}})\in \mathsf{HOD}$.

Suppose $\pi_3(\overline{\gamma})=\gamma$. Take  $\delta>|\mathsf{V}_{\gamma+\omega+1}|$ such that $\delta<\lambda$ and $\delta$ is not $\omega$-strongly measurable in $\mathsf{HOD}$. Let $S^{\delta}_{\omega}=\lbrace\alpha<\delta\mid cf(\alpha)=\omega\rbrace$.  Since  $\delta$ is not $\omega$-strongly measurable in $\mathsf{HOD}$, there exists $\langle S_{\alpha}\mid\alpha< |\mathsf{V}_{\gamma+\omega}|\rangle\in \mathsf{HOD}$ which is  a  partition of $S^{\delta}_{\omega}$ into  stationary sets in $\delta$. Let $a=\langle\gamma, \delta,  (S_{\alpha}\mid\alpha< |\mathsf{V}_{\gamma+\omega}|)\rangle$. Note that $a\in \mathsf{V}_{\lambda}\cap \mathsf{HOD}$. Since $\mathsf{HOD}\cap \mathsf{V}_{\lambda}=\mathsf{HOD}^{\mathsf{V}_{\lambda}}, a$ is definable in $\mathsf{V}_{\lambda}$. Let  $\pi_3(\overline{\delta})=\delta$ and $\pi_3(\langle \overline{S_{\alpha}}\mid\alpha< |\mathsf{V}_{\overline{\gamma}+\omega}|\rangle)=\langle S_{\alpha}\mid\alpha< |\mathsf{V}_{\gamma+\omega}|\rangle$.  Then $\langle \overline{S_{\alpha}}\mid\alpha<  |\mathsf{V}_{\overline{\gamma}+\omega}|\rangle\in \mathsf{HOD}^{\mathsf{V}_{\lambda_1}}$ is a partition of $S^{\overline{\delta}}_{\omega}$  into  stationary sets in $\overline{\delta}$. For $\eta<\overline{\delta}$ such that $cf(\eta)>\omega$, let $\sigma_{\eta}=\lbrace \alpha< |\mathsf{V}_{\overline{\gamma}+\omega}|: \overline{S_{\alpha}}\cap\eta$ is stationary in $\eta\rbrace$. Note that $\langle \sigma_{\eta}: \eta<\overline{\delta}, cf(\eta)>\omega\rangle\in \mathsf{HOD}^{\mathsf{V}_{\lambda_1}}$. Let $\langle \tau_{\eta}: \eta<\delta, cf(\eta)>\omega\rangle=\pi_3(\langle \sigma_{\eta}: \eta<\overline{\delta}, cf(\eta)>\omega\rangle)$. For each  $\eta<\delta$ such that  $cf(\eta)>\omega, \tau_{\eta}=\lbrace \alpha< |\mathsf{V}_{\gamma+\omega}|:  S_{\alpha}\cap\eta$ is stationary in $\eta\rbrace$. Then $\langle \tau_{\eta}: \eta<\delta, cf(\eta)>\omega\rangle\in \mathsf{HOD}^{\mathsf{V}_{\lambda}}=\mathsf{HOD}\cap \mathsf{V}_{\lambda}$. Let $\eta_0=sup\lbrace \pi_3(\xi)\mid\xi< \overline{\delta}\rbrace$. By Lemma \ref{isolated lemma}, we have $\tau_{\eta_0}=\{\pi_3(x): x\in |\mathsf{V}_{\overline{\gamma}+\omega}|\}$.

To show that $\pi_3\upharpoonright (\mathsf{HOD}\cap \mathsf{V}_{\overline{\gamma}})\in \mathsf{HOD}$, it suffices to show that $\lbrace \pi_3(x): x\in \mathsf{HOD}\cap \mathsf{V}_{\overline{\gamma}} \rbrace\in \mathsf{HOD}$. Take $j\in \mathsf{HOD}\cap  \mathsf{V}_{\lambda_1}$ such that $j: \theta\rightarrow \mathsf{HOD}\cap \mathsf{V}_{\overline{\gamma}}$ is a surjection for some $\theta< |\mathsf{V}_{\overline{\gamma}+\omega}|$. For $x\in \mathsf{HOD}\cap \mathsf{V}_{\overline{\gamma}}, \pi_3(x)=\pi_3(j(\xi))=\pi_3(j)(\pi_3(\xi))$ for some $\xi\in\theta$. Since $\tau_{\eta_0}\in \mathsf{HOD}$, we have $\pi_3\upharpoonright |\mathsf{V}_{\overline{\gamma}+\omega}| \in \mathsf{HOD}$. Since $j\in \mathsf{HOD}\cap  \mathsf{V}_{\lambda_1}$, $\pi_1(\mathsf{HOD}\cap \mathsf{V}_{\lambda_1}) = \mathsf{HOD} \cap \mathsf{V}_{\lambda_2}$ and $\pi_2(\mathsf{HOD}\cap \mathsf{V}_{\lambda_2}) = \mathsf{HOD} \cap \mathsf{V}_{\kappa}$, we have  $\pi_3(j)=\pi_2(\pi_1(j))\in \mathsf{HOD}$. Since $\pi_3(j)\in \mathsf{HOD}$ and  $\pi_3\upharpoonright |\mathsf{V}_{\overline{\gamma}+\omega}| \in \mathsf{HOD}$, we have $\pi_3\upharpoonright (\mathsf{HOD}\cap \mathsf{V}_{\overline{\gamma}})\in \mathsf{HOD}$.
\end{proof}

There are four key points in the proof of Theorem \ref{main result of the paper}:  (1) from Theorem \ref{equivalence of supercompact} we get an embedding $\pi_3: \mathsf{V}_{\lambda_1+1} \to \mathsf{V}_{\lambda+1}$  such that $crit(\pi_3)=\kappa_1$ and $\pi_3(\mathsf{HOD}\cap \mathsf{V}_{\lambda_1}) \subseteq \mathsf{HOD}$; (2) from the $\mathsf{HOD}$ Hypothesis, we can get a partition $\langle S_{\alpha}\mid\alpha< |\mathsf{V}_{\gamma+\omega}|\rangle$ in $\mathsf{HOD}$ of $\delta$ into stationary subsets; (3) from Lemma \ref{isolated lemma}, $\tau_{\eta_0}=\{\pi_3(x): x\in |\mathsf{V}_{\overline{\gamma}+\omega}|\}$, and from $\langle S_{\alpha}\mid\alpha< |\mathsf{V}_{\gamma+\omega}|\rangle\in\mathsf{HOD}$, we can show that $\pi_3\upharpoonright |\mathsf{V}_{\overline{\gamma}+\omega}| \in \mathsf{HOD}$; (4) from $\pi_3(\mathsf{HOD}\cap \mathsf{V}_{\lambda_1}) \subseteq \mathsf{HOD}$ and   $\pi_3\upharpoonright |\mathsf{V}_{\overline{\gamma}+\omega}| \in \mathsf{HOD}$, by a standard argument, we can show that $\pi_3\upharpoonright (\mathsf{HOD}\cap \mathsf{V}_{\overline{\gamma}})\in \mathsf{HOD}$.

\begin{corollary}\label{the main theorem}
Suppose $\kappa$ is supercompact and the $\mathsf{HOD}$ Hypothesis holds. Then $\mathsf{V}_{\kappa}\models$ there is a proper class of regular cardinals which are measurable in $\mathsf{HOD}$.
\end{corollary}

The following Local University Theorem follows from proof of Theorem \ref{main result of the paper}, and is a reformulation of Woodin's original version announced in \cite{Woodin 5}.\footnote{I would like to thank Woodin for the communication with the author  about his work on Local Universality Theorem.}

\begin{corollary}\label{Local Universality Theorem}
(Woodin, Local Universality Theorem)\quad Suppose $\kappa$ is supercompact and the $\mathsf{HOD}$ Hypothesis holds. Then for each $\alpha<\kappa$, there exists an elementary embedding $j: \mathsf{V}_{\lambda+1}\rightarrow \mathsf{V}_{j(\lambda)+1}$ such that
\begin{enumerate}[(1)]
  \item $crit(j)=\overline{\kappa}$, $\alpha<\overline{\kappa}<\lambda<\kappa$ and $j(\lambda)<\kappa$;
  \item  $j\upharpoonright (\mathsf{HOD}\cap \mathsf{V}_{\lambda})\in \mathsf{HOD}$ and
  \item $j(\mathsf{HOD}\cap \mathsf{V}_{\lambda})=\mathsf{HOD}\cap \mathsf{V}_{j(\lambda)}$.
\end{enumerate}
\end{corollary}

From \cite{Woodin 5}, Woodin essentially proved that if $\delta$ is $N$-supercompact and $N$ is a weak extender model for $\delta$-supercompact, then any measurable cardinal $\kappa\geq\delta$ is measurable in $N$. As a corollary, if $\delta$ is $\mathsf{HOD}$-supercompact and the $\mathsf{HOD}$ Hypothesis holds, then any measurable cardinal $\kappa\geq\delta$ is measurable in $\mathsf{HOD}$. Comparing  this result with Theorem \ref{main result of the paper}, and Global Universality Theorem with Local Universality Theorem, we can see the huge difference between $\mathsf{HOD}$-supercompact cardinals and supercompact cardinals under the $\mathsf{HOD}$ Hypothesis even if $\mathsf{HOD}$-supercompact cardinals and supercompact cardinals seem to be close in the large cardinal hierarchy: under the assumption of the $\mathsf{HOD}$ Hypothesis and $\mathsf{HOD}$-supercompact cardinals, large cardinals in $\mathsf{V}$ are reflected to be large cardinals in $\mathsf{HOD}$ in a global way; however, under the assumption of the $\mathsf{HOD}$ Hypothesis and supercompact cardinals, large cardinals in $\mathsf{V}$ are reflected to be large cardinals in $\mathsf{HOD}$ in a local way.

\section{Questions}

Theorem \ref{big thm from woodin in SEM} and Theorem \ref{equi be HC and MC} have established the equivalence of the $\mathsf{HOD}$ Hypothesis under the assumption of $\mathsf{HOD}$-supercompact cardinals and extendible cardinals. A natural question is whether we can establish the equivalence of the $\mathsf{HOD}$ Hypothesis only assuming supercompact cardinals. Especially, if $\kappa$ is supercompact, whether the $\mathsf{HOD}$ Hypothesis is the equivalent to the statement: for each $\alpha<\kappa$,  there exists $\gamma$ such that $\alpha<\gamma<\kappa$ and $\gamma$ is measurable in $\mathsf{HOD}$.\footnote{I would like to thank the referee for pointing out this question to me.} It seems for me the $\mathsf{HOD}$ Hypothesis expresses the global property of large cardinals in $\mathsf{HOD}$. In this paper, we prove the forward direction. I conjecture that the backward direction does not hold.  The difficulty in proving this conjecture comes from the difficulty in forcing the failure of the $\mathsf{HOD}$ Hypothesis. Under the assumption of only supercompactness, as far as we know, we do not know any equivalence of the $\mathsf{HOD}$ Hypothesis. Woodin conjectured in \cite{Woodin 5} that if $\delta$ is supercompact then the $\mathsf{HOD}$ Hypothesis is equivalent to the existence of a weak extender model $N$ for $\delta$ supercompact such that $N \subseteq \mathsf{HOD}$, which as far as we know is an open problem.


\begin{thebibliography}{99}
\bibitem{joint work with Joel}
{Yong Cheng, Sy-David Friedman and Joel David Hamkins.} \emph{Large cardinals need not be large in $\mathsf{HOD}$}. Annals of Pure and Applied Logic, vol. 166, iss. 11, pp. 1186-1198, 2015.


\bibitem{Kunen inconsistency}
Joel David Hamkins, Greg Kirmayer, Norman Lewis Perlmutter. \emph{Generalizations of the Kunen inconsistency}. Annals of Pure and Applied Logic 163 (2012) 1872-1890.


\bibitem{Jech}
{Thomas Jech,} \emph{Set theory, } Springer-Verlag 2003. Third millennium edition.


\bibitem{Set theory}
{Akihiro Kanamori.} \emph{The Higher Infinite: Large Cardinals in Set Theory from
Their Beginnings}. Second edition, (Berlin: Springer).




\bibitem{Woodin 2}
W.\ Hugh Woodin, \emph{Suitable extender models I}, J.\ Math.\ Log.\ 10, 101 (2010).

\bibitem{Woodin 3}
W.Hugh Woodin. \emph{Suitable extender models II: Beyond $\omega$-huge}. J. Math. Log, 11, 115 (2011).

\bibitem{The Future Trichotomy}
W.Hugh Woodin. \emph{The Future Trichotomy}, July 15, 2013. Slide avaliable at http://estcongress.org/Slides/Woodin.pdf


\bibitem{The HOD Dichotomy}
W.Hugh Woodin. \emph{The $\mathsf{HOD}$ Dichotomy}, February 20, 2011. Slide avaliable at http://logic.harvard.edu/MAMLS


\bibitem{Woodin 1}
W.\ Hugh Woodin, Jacob Davis and Daniel Rodr\'iguez, \emph{The $\mathsf{HOD}$ Dichotomy}, Appalachian Set Theory 2006-2012, pp. 397-419, London Mathematical Society Lecture Note Series (No. 406).


\bibitem{Woodin new}
W.Hugh Woodin, \emph{In search of Ultimate-L: The 19th Midrasha Mathematicae Lectures}, The Bulletin of Symbolic Logic, Vol. 23, No. 1 (MARCH 2017), pp. 1-109.



\bibitem{Woodin 5}
W.Hugh Woodin. Personal email communication with W.Hugh Woodin.
\end{thebibliography}
\end{document}